\documentclass[a4paper]{amsart}

\usepackage{amsfonts,amssymb,mathtools}
\usepackage{textcomp, color}

\theoremstyle{plain}

\newtheorem*{thmA}{Theorem A}

\newtheorem{thm}{Theorem}[section]
\newtheorem{lem}[thm]{Lemma}

\newtheorem{cor}[thm]{Corollary}

\theoremstyle{definition}

\newtheorem{rmk}[thm]{Remark}

\newcommand{\F}{\mathbb{F}}

\DeclareMathOperator{\cl}{cl}

\begin{document}

\title[Metabelian thin Beauville $p$-groups]{Metabelian thin Beauville $p$-groups}

\author[N.\ Gavioli]{Norberto Gavioli}
\address{Universit\`a degli Studi dell'Aquila\\
	L'Aquila, Italy	}
\email{gavioli@univaq.it}

\author[\c{S}.\ G\"ul]{\c{S}\"ukran G\"ul}
\address{Department of Mathematics\\ Middle East Technical University\\
	06800 Ankara, Turkey}
\email{gsukran@metu.edu.tr}

\author[C.\ Scoppola]{Carlo Scoppola}
\address{Universit\`a degli Studi dell'Aquila\\
		L'Aquila, Italy}
\email{scoppola@univaq.it}

\keywords{Beauville groups; metabelian thin $p$-groups \vspace{3pt}}

\thanks{The second author is supported by the Spanish Government, grant MTM2014-53810-C2-2-P, and the Basque Government, grant IT974-16}.

\begin{abstract}
	A non-cyclic finite $p$-group $G$ is said to be \emph{thin} if every normal subgroup of $G$ lies between two consecutive terms of the lower central series and  $|\gamma_i(G):\gamma_{i+1}(G)|\le p^2$ for all $i\geq 1$.  In this paper, we  determine Beauville structures in metabelian thin $p$-groups.
\end{abstract}

\maketitle

\section{Introduction}

 \emph{Beauville surfaces} \cite[page 159]{bea} are complex surfaces of general type constructed from two orientable regular hypermaps of genus at least $2$, with the same automorphism groups.  A finite group which arises as an automorphism group linked to a Beauville surface is called a \emph{Beauville group}.

A group-theoretical formulation of Beauville groups can be given as follows. Given a finite group $G$ and a couple of elements $x,y \in G$, we define
\[
\Sigma(x,y)
=
\bigcup_{g\in G} \,
\Big( \langle x \rangle^g \cup \langle y \rangle^g \cup \langle xy \rangle^g \Big),
\]
that is, the union of all subgroups of $G$ which are conjugate to $\langle x \rangle$, to 
$\langle y \rangle$ or to $\langle xy \rangle$. Then $G$ is a Beauville group if and only if the following conditions hold:
\begin{enumerate}
	\item $G$ is a $2$-generator group.
	\item There exists a pair of generating sets $\{x_1,y_1\}$ and $\{x_2,y_2\}$ of $G$ such that 
	$\Sigma(x_1,y_1) \cap \Sigma(x_2,y_2)=1$.
\end{enumerate}
Then $ \{x_1,y_1\}$ and $\{x_2,y_2\}$ are said to form a \emph{Beauville structure\/} for $G$.

By using this characterization, one can study whether a given finite group is a Beauville group.
 For example, Catanese \cite{cat} showed that a finite abelian group is a Beauville group if and only if it is isomorphic to $C_n\times C_n$, with $n>1$ and $\gcd(n,6)=1$. On the other hand, a remarkable result, proved independently by Guralnick and Malle \cite{GM} and by Fairbairn, Magaard and Parker \cite{FMP} in 2012, is that every non-abelian finite simple group other than $A_5$ is a Beauville group.
 
 If $p$ is a prime, not much was known about Beauville $p$-groups until very recently (see \cite{BBF} and \cite{bos}). In \cite[Theorem 2.5]{FG}, Fern\'andez-Alcober and G\"ul extended Catanese's criterion  in the case of $p$-groups from abelian groups to a much wider family of groups, including all $p$-groups having a `good behaviour' with respect to taking powers, and in particular groups of class $<p$. 
 
 A non-cyclic finite $p$-group $G$ is said to be \emph{thin} if every normal subgroup of $G$ lies between two consecutive terms of the lower central series and  $|\gamma_i(G):\gamma_{i+1}(G)|\le p^2$ for all $i\geq 1$. Indeed, these groups are $2$-generator. Thus it is natural to ask whether they are Beauville or not. Furthermore, the study of thin $p$-groups is also motivated by the fact that they usually give examples of groups whose power structures are not so well-behaved.  Well-known examples of thin $p$-groups are $p$-groups of maximal class and quotients of the Nottingham group. In  \cite{FG}, all Beauville quotients of the Nottingham group were determined. Thanks to the ill-behaved power structure, the first explicit infinite family of Beauville $3$-groups was given by considering quotients of the Nottingham group over $\F_3$. In \cite{FG2},  the existence of Beauville structures in the most important families of groups of maximal class, in particular in metabelian $p$-groups of maximal class, was studied.
 
 The goal of this paper is to complete the study of Beauville structures in \break 
 metabelian thin $p$-groups. In this paper, we will exclude $p$-groups of maximal class from our consideration of thin groups. This means, in particular, that the prime $2$ is excluded \cite[Theorem III.11.9]{hup}.
 Then according to Theorem A in \cite{BCS}, if $G$ is a metabelian thin $p$-group,  then $\cl(G)\leq p+1$. If the class is  $<p$, then the existence of Beauville structures can be determined by using Corollary 2.6 in \cite{FG}. Thus we restrict to groups of class $p$ or $p+1$. The main result of this paper is as follows.

    \begin{thmA}
		Let $G$ be a metabelian thin $p$-group of class $p$ or $p+1$ for $p\geq 5$.
		Then there are four cases in which there is a Beauville structure:
		\begin{enumerate}
			\item
			$\cl(G)=p$ and $|\gamma_p(G)|=p^2$.
			\item
			$\cl(G)=p+1$.
			\item
			$\cl(G)=p$, $|\gamma_p(G)|=p$ and $G^p=\gamma_{p-1}(G)$.
			\item
			$\cl(G)=p$, $|\gamma_p(G)|=p$, $G^p=\gamma_p(G)$ and $G$ has at least three maximal subgroups of exponent
			$p$.
		\end{enumerate}
	\end{thmA}

	\vspace{10pt}

	\noindent
	\textit{Notation.\/}
	We use standard notation in group theory. If $G$ is a group, $H,K \leq G$ and $N\unlhd G$, then $H\equiv K \pmod{N}$ means $HN/N=KN/N$. If $p$ is a prime, then we write $G^{p^i}$ for the subgroup generated by all powers $g^{p^i}$ as $g$ runs over $G$, and $\Omega_i(G)$ for the subgroup generated by the elements of $G$ of order at most $p^i$. The exponent of $G$, denoted by $\exp G$, is the maximum of the orders of all elements of $G$.

\section{Proof of the main result}


In this section, we give the proof of Theorem A. We begin by giving some properties of metabelian thin $p$-groups.  Firstly, we recall the following more general result of Meier-Wunderli. If $G$ is a metabelian $2$-generator $p$-group, then 
\begin{equation}
\label{wunderli}
G^p\geq \gamma_p(G)
\end{equation}
(see \cite{wun}, Theorem 3).

Observe that the only thin abelian $p$-group is the elementary abelian group of order $p^2$ and we refer to its lattice of normal subgroups as a \emph{diamond\/}.  It then follows that if $G$ is thin, then $G/G'$ is elementary abelian of order $p^2$, and  hence $G$ is $2$-generator. Also,  the lower and upper central series of a thin $p$-group coincide \cite[Corollary 2.2]{BCS}.

Note that if $G$ is a thin $p$-group and $g\in\gamma_i(G)\smallsetminus \gamma_{i+1}(G)$, then
\begin{equation}
\label{covering property}
\gamma_{i+1}(G)=[g,G]\gamma_{i+2}(G)
\end{equation}
(see \cite{BCS}, Lemma 2.1).

Now let $G$ be a metabelian thin $p$-group. By Theorem A in \cite{BCS}, we have the following:
\begin{enumerate}
	\item $\gamma_{p+1}(G)$ is cyclic and $\gamma_{p+2}(G)=1$.
	\item  The lattice of normal subgroups of $G$ consists of a diamond on top, followed by a chain of length $1$, at most $p-2$ diamonds, plus possibly another chain of length $1$.
\end{enumerate}

As a consequence, $\cl(G)\leq p+1$, and $|G|\leq p^{2p}$.  We next recall the power structure of a metabelian thin $p$-group.

\begin{lem}
	\label{place of G^p}
	Let $G$ be a metabelian thin $p$-group, and $l$ be the largest integer such that $G^p\leq \gamma_{l}(G)$. Then $3\leq l \leq p$, $\gamma_{l+1}(G)$ is cyclic, $\gamma_{l+2}(G)=1$ and $\gamma_2(G)^p\leq \gamma_{l+1}(G)$.
\end{lem}

\begin{proof}
	See the proofs of Lemmas 1.2, 1.3 and 3.3 in \cite{BCS}.
\end{proof}

The next corollary follows directly from Lemma \ref{place of G^p}.

\begin{cor}
	Let $G$ be a metabelian thin $p$-group. Then $|G^p|\leq p^3$.
\end{cor}

\begin{lem}
	\label{order of G^p}
	Let $G$ be a metabelian thin $p$-group such that its lattice of normal subgroups ends with a chain. Then the order of $G^p$ cannot be $p^2$.
\end{lem}

\begin{proof}
	If $G$ is of class $p+1$, then $G^p=\gamma_p(G)$, and hence $|G^p|=p^3$. Thus we assume that $\cl(G)=c\leq p$.  Next observe that if $M$ is a maximal subgroup of $G$, then $\cl(M)<c\leq p$, and so $M$ is regular.
	
	Now suppose, on the contrary, that $|G^p|=p^2$. Consider the quotient group $\overline{G}= G/ \gamma_c(G)$, which is regular. Then $|\overline{G}^p|=p$, and hence $|\overline{G}: \Omega_1(\overline{G})|=p$. Write $\Omega_1(\overline{G})=M/\gamma_c(G)$ for some maximal subgroup $M$ of $G$. Since $\overline{G}$ is regular,  $\exp \ \Omega_1(\overline{G})=p$, and hence $M^p\leq \gamma_c(G)$. This implies that 
	 $|M^p|= |M:\Omega_1(M)|\leq p$ as $M$ is regular. It then follows that 
	  $G'\leq \Omega_1(M)$, and thus $\exp \ G'=p$.
	
	On the other hand, if $M$ is an arbitrary maximal subgroup of $G$, we have $G'\leq M$ and since $\exp G' =p$, we get $G' \leq \Omega_1(M) \leq M$. Then $|M^p|=|M:\Omega_1(M)|\leq p$. Since $G$ is thin, this implies that $M^p \leq \gamma_c(G)$ for any maximal subgroup $M$. But $G^p= \langle M^p \mid \text{$M$ maximal in $G$} \rangle \leq \gamma_c(G)$. Thus $|G^p|\leq p$, which is a contradiction.
\end{proof}

We finally recall a commutator relation between the generators of $G$. More specifically, if $G$ is a metabelian thin $p$-group, then to every $x\in G\smallsetminus G'$ there corresponds a $y$ such that $G=\langle x,y \rangle$ and 
\begin{equation}
\label{quadratic-rel}
[y,x,x,x] \equiv [y,x,y,y]^h \pmod{\gamma_5(G)}.
\end{equation}
where $h$ is a quadratic non-residue modulo $p$ \cite[Theorem B]{BCS}.

Before we proceed to prove Theorem A, we will determine which metabelian thin $3$-groups are Beauville groups.

\begin{thm}
\label{p=3}
A metabelian thin $3$-group is a Beauville group if and only if it is one of $SmallGroup(3^5,3)$, $SmallGroup(3^6,34)$, or  $SmallGroup(3^6,37)$.
\end{thm}

\begin{proof}
Let $G$ be a metabelian thin $3$-group. Then  $|G|\leq 3^6$. Note that the smallest Beauville $3$-group $S=\texttt{SmallGroup}(3^5,3)$ is of order $3^5$,  and it is the only Beauville $3$-group of that order \cite{BBF}. Furthermore, by using the computer algebra system GAP \cite{GAP}, it can be seen that $S$ is  metabelian thin. Thus, if $|G|=3^5$, then $G$ is a Beauvile group if and only if $G\cong S$. 
We next assume that $|G|=3^6$. It has been shown in \cite{BBF} that there are only three Beauville $3$-groups of order $3^6$, namely  $S=\texttt{SmallGroup}(3^6,n)$ for $n=34,37,40$. However, if $n=40$ then $S$ is not thin since  $|Z(S)|=9$ and thus $Z(S)\neq \gamma_4(S)$. On the other hand, if $n=34$ or $37$ then again by using the computer algebra system GAP, one can show that $S$ is a metabelian thin $3$-group. Consequently, $G$ is a Beauville group if and only if $G\cong S$ for $n=34$ or $37$.
\end{proof}

Thus we assume that $p\geq 5$. Let $G$ be a metabelian thin $p$-group with $\cl(G)=p$ or $p+1$. Then we have three cases: $G$ is of class $p+1$, or  is of class $p$ and $|\gamma_p(G)|=p^2$, or is of class $p$ and $|\gamma_p(G)|=p$. In the first two cases,  we have $G^p \leq \gamma_p(G) $. It then follows from (\ref{wunderli}) that $G^p=\gamma_p(G)$. Note that we also have $\gamma_2(G)^p \leq \gamma_{p+1}(G)$, by Lemma \ref{place of G^p}. On the other hand, in the last case if $l$ is the largest integer satisfying $G^p\leq \gamma_l(G)$, then $l=p-1$ or $p$ and hence $\gamma_p(G)\leq G^p\leq \gamma_{p-1}(G)$.

Our first step is to calculate the $p$th powers of $x^ty$ modulo $\gamma_{p+1}(G)$ for all $0\leq t\leq p-1$ if $G=\langle x,y \rangle$ and $\gamma_2(G)^p \leq \gamma_{p+1}(G)$. To this purpose, we need the following lemma.

\begin{lem}\textup{\cite[Lemma 6]{mie} }
\label{power-rel}
Let $G$ be a metabelian $p$-group and $x,y \in G$. Set $\sigma_1=y$ and $\sigma_i=[\sigma_{i-1},x]$ for $i\geq 2$. Then
\[
(xy)^p= x^py^p\sigma_2^{\binom{p}{2}}\dots \sigma_p^{\binom{p}{p}}z,
\]
where
\[
z= \prod_{i=1}^{p-1}\prod_{j=1}^{p-1}[\sigma_{i+1},_j \sigma_1]^{C(i,j)},
\]
and
\[
C(i,j)=\sum_{k=1}^{p-1}\binom{k}{i}\binom{k}{j}.
\]
\end{lem}

\begin{lem}
Let $G$ be a metabelian thin $p$-group such that $\gamma_2(G)^p \leq \gamma_{p+1}(G)$. If  $x$ and $y$ are the generators of $G$ satisfying (\ref{quadratic-rel}), then for all $0\leq t\leq p-1$
\begin{equation}
\label{congruence of pth powers}
(x^ty)^p\equiv (x^p)^ty^p [y,x,_{p-2}y]^{\frac{-2t}{1-ht^2}}[y,x,_{p-3}y,x]^{\frac{2t^2}{1-ht^2}} \pmod{\gamma_{p+1}(G)}.
\end{equation}
\end{lem}

\begin{proof}
By Lemma \ref{power-rel}, we have
\[
(x^ty)^p=(x^p)^ty^p[y,x^t]^{\binom{p}{2}}\dots [y,_{p-1}x^t]^{\binom{p}{p}}\prod_{i=1}^{p-1}\prod_{j=1}^{p-1}[y,_i x^t,_j y]^{C(i,j)}.
\]
Since $\gamma_2(G)^p \leq \gamma_{p+1}(G)$, it then follows that
\[
(x^ty)^p \equiv (x^p)^ty^p[y,_{p-1}x^t] \prod_{\substack{1\le i,j \\ \ i+j\le p-1}}[y,_i x^t,_j y]^{C(i,j)} \pmod {\gamma_{p+1}(G)}.
\]

Note that for $i+j>0$, $C(i,j)$ is the coefficient of $u^iv^j$ in $\sum_{k=0}^{p-1}(1+u)^k(1+v)^k$, where
\[
\sum_{k=0}^{p-1}(1+u)^k(1+v)^k \equiv \bigl((u+v)+uv\bigr)^{p-1} \pmod{p}.
\]

In the previous expression the monomials of total degree less than $p$ appear only in $(u+v)^{p-1}\equiv \sum_{r=0}^{p-1}(-1)^r u^r v^{p-r-1} \pmod{p}$, and hence
\begin{equation*}
C(i,j) \equiv
\begin{cases}
0  \pmod{p}& \text{if \ $i+j<p-1$,}\\
(-1)^{i} \pmod{p} & \text{if \ $i+j=p-1$. }
\end{cases}
\end{equation*}

Thus the condition $\gamma_2(G)^p\leq \gamma_{p+1}(G)$ implies that
\[
(x^ty)^p \equiv (x^p)^ty^p\prod_{i=1}^{p-1}[y,_i x^t,_{p-i-1} y]^{(-1)^i}  \pmod {\gamma_{p+1}(G)}.
\]

On the other hand, notice that for $1\leq t \leq p-1$
\begin{equation}
\label{quad-rel for t}
[y,x^t,x^t,x^t]\equiv [y,x^t,y,y]^{ht^2}\pmod{\gamma_5(G)}.
\end{equation}
Since $G$ is metabelian, for every $a,b\in G$ and $c\in G'$ we have $[c,a,b]=[c,b,a]$, and this, together with (\ref{quad-rel for t}), yields that
\begin{equation*}
	[y,_i x^t,_{p-i-1} y]^{(-1)^i} \equiv
	\begin{cases}
	[y,x^t,_{p-2}y]^{-(ht^2)^{s-1}} \pmod{\gamma_{p+1}(G)}& \hspace{-1em}\text{ if \ $i=2s-1$ },\\
	[y,x^t,_{p-3}y,x^t]^{(ht^2)^{s-1}} \pmod{\gamma_{p+1}(G)}& \hspace{-1em}\text{ if \ $i=2s$ },
	\end{cases}
\end{equation*}
and hence
\begin{equation*}
(x^ty)^p \equiv (x^p)^ty^p \left( [y,x,_{p-2}y]^{-t}[y,x,_{p-3}y,x]^{t^2} \right)^
{\sum_{s=1}^{(p-1)/2}(ht^2)^{s-1}} \pmod{\gamma_{p+1}(G)}.
\end{equation*}

Note that since $h$ is a quadratic non-residue, we have $\sum_{s=1}^{(p-1)/2}(ht^2)^{s-1}=\frac{2}{1-ht^2}$. Consequently, we get
\begin{equation*}
(x^ty)^p \equiv (x^p)^ty^p [y,x,_{p-2}y]^{\frac{-2t}{1-ht^2}}[y,x,_{p-3}y,x]^{\frac{2t^2}{1-ht^2}} \pmod{\gamma_{p+1}(G)},
\end{equation*}
for $0\leq t\leq p-1 $, as desired.
\end{proof}

\begin{lem}
\label{collision}
Let $G$ be a metabelian thin $p$-group such that $|\gamma_p(G)|\geq p^2$ and let $x$ and $y$ be the generators of $G$ satisfying (\ref{quadratic-rel}). Then for every $t_0\in \{0,1\dots,p-1\}$ there exist at most three distinct $t\in \{0,1\dots,p-1\}$ such that
\begin{equation}
\label{coincidence}
\langle (x^{t_0}y)^p\rangle \equiv \langle (x^ty)^p\rangle \pmod{\gamma_{p+1}(G)}.
\end{equation}
\end{lem}

\begin{proof}
Since $|\gamma_p(G)|\geq p^2$, we have $G^p=\gamma_p(G)$ and  $\gamma_2(G)^p \leq \gamma_{p+1}(G)$, by Lemma \ref{place of G^p}. Also note that $|\gamma_p(G): \gamma_{p+1}(G)|=p^2$. Notice that, as a consequence of (\ref{covering property}), $l=[y,x,_{p-2}y]$ and $m=[y,x,_{p-3}y,x]$ are linearly independent modulo $\gamma_{p+1}(G)$. Thus $(l,m)$ is a basis of $\gamma_p(G)$ modulo $\gamma_{p+1}(G)$. If we set 
$x^p \equiv l^{\alpha}m^{\beta} \pmod{\gamma_{p+1}(G)}$ and 
$y^p \equiv l^{\gamma} m^{\delta} \pmod{\gamma_{p+1}(G)}$ for some $\alpha, \beta, \gamma, \delta \in \F_p$, then by (\ref{congruence of pth powers}), we have
\begin{equation}
\label{congruence wrto basis}
(x^ty)^p \equiv l^{\gamma+ \alpha t- \frac{2t}{1-ht^2}} \ \ m^{\delta+\beta t+ \frac{2t^2}{1-ht^2}} 
\pmod{\gamma_{p+1}(G)}.
\end{equation}
 Observe that as rational functions in $t$, neither $f(t)=\gamma+ \alpha t- \frac{2t}{1-ht^2}$ nor  
$g(t)=\delta+\beta t+ \frac{2t^2}{1-ht^2}$ are zero.

We now fix $t_0 \in \{0,1\dots,p-1\}$. Then (\ref{coincidence}) holds if and only if there exists $\lambda \in \F_p^{*}$ such that
\begin{gather*}
f(t)= \lambda f(t_0)
\quad
\text{and}
\quad
g(t)= \lambda g(t_0).
\end{gather*}
If $f(t_0)=0$ or $g(t_0)=0$, then we have  $f(t)=0$ or $g(t)=0$, that is $(1-ht^2)(\gamma+\alpha t)-2t=0$ or $(1-ht^2)(\delta+ \beta t)+2t^2=0$. Otherwise, we have $\frac{f(t)}{f(t_0)}=\frac{g(t)}{g(t_0)}$. Then $g(t_0)f(t)-f(t_0)g(t)=0$, that is 
\[
g(t_0)\big((1-ht^2)(\gamma+\alpha t)-2t \big)-f(t_0) \big((1-ht^2)(\delta+ \beta t)+2t^2\big)=0,
\]
which is a polynomial in $t$ of degree $\leq 3$. Thus in every case, there are at most three distinct  $t\in \{0,1\dots,p-1\}$ such that $\langle (x^{t_0}y)^p\rangle\equiv \langle (x^ty)^p \rangle \pmod{\gamma_{p+1}(G)}$.
\end{proof}

\begin{lem}
\label{companion}
Let $G$ be a metabelian thin $p$-group such that $\gamma_2(G)^p \leq \gamma_{p+1}(G)$. If $M$ is a maximal subgroup of $G$ and $a,b \in M\smallsetminus G'$, then $\langle a \rangle^p \equiv \langle b\rangle^p \pmod{\gamma_{p+1}(G)}$.
\end{lem}

\begin{proof}
If we write $b= a^ic$ for some $c\in G'$ and for some integer $i$ not divisible by $p$, then by the Hall-Petrescu collection formula (see \cite[III.9.4]{hup}), we have 
\[
(a^ic)^p= a^{pi}c^p c_2^{\binom{p}{2}}c_3^{\binom{p}{3}}\dots c_p,
\]
where $c_j \in \gamma_j(\langle a,c \rangle)\leq \gamma_{j+1}(G)$. Thus $(a^ic)^p \equiv a^{pi} \pmod{\gamma_{p+1}(G)}$, and hence $\langle a \rangle^p \equiv \langle b\rangle^p \pmod{\gamma_{p+1}(G)}$.
\end{proof}

\begin{rmk}
	If we replace $x$ with $x^{*}$, where $x^{*}\in G\smallsetminus G'$ is not a power of $x$, there exists a corresponding $y^{*}$ satisfying (\ref{quadratic-rel}). Then $x\in \langle (x^{*})^{t_0}y^{*}, G'\rangle\smallsetminus G'$ for some $0\leq t_0\leq p-1$, and according to Lemma \ref{companion}, we have $\langle x^p \rangle \equiv \langle((x^{*})^{t_0}y^{*})^p \rangle\pmod{\gamma_{p+1}(G)}$. It then follows from Lemma \ref{collision} that there exist at most three distinct $t\in \{0,1\dots,p-1\}$ such that $\langle x^p\rangle\equiv \langle (x^ty)^p \rangle \pmod{\gamma_{p+1}(G)}$.
\end{rmk}

The following corollary is an immediate consequence of Lemmas \ref{collision} and \ref{companion}.

\begin{cor}
\label{number of coincidence}
Let $G$ be a metabelian thin $p$-group such that $|\gamma_{p}(G)|\ge p^2$. If  $M$ is a maximal subgroup of $G$, then there exist at most two maximal subgroups $M_1$, $M_2$ different from $M$  such that $M^p\equiv M_1^p \equiv M_2^p \pmod{\gamma_{p+1}(G)}$.
\end{cor}

Before we present the main result, we also need the following remark.

\begin{rmk}
	\label{modifying generators}
	Let $G$ be a finite $2$-generator $p$-group. Then we can always find elements $x,y \in G\smallsetminus \Phi(G)$ such that $x,y$ and $xy$ fall into the given three maximal subgroups of $G$. Let $M_1$, $M_2$ and $M_3$ be three maximal subgroups of $G$. Choose $x\in M_1 \smallsetminus \Phi(G)$ and $y \in M_2 \smallsetminus \Phi(G)$.  Since each element in the set $\{xy^j \mid 1\leq j \leq p-1 \}$ falls into different maximal subgroups, there exists $1\leq j \leq p-1$ such that $xy^j \in M_3\smallsetminus \Phi(G)$. Thus if we put $x*=x$ and $y*=y^j$, then elements in the triple $\{x^{*},y^{*}, x^{*}y^{*}\}$ fall into the given three maximal subgroups.
\end{rmk}

We are now ready to prove Theorem A. We deal separately with the cases in the theorem.

\begin{thm}
\label{class p}
Let $G$ be a metabelian thin $p$-group with $\cl(G)=p$ such that $|\gamma_p(G)|=p^2$, where $p\geq 5$. Then $G$ has a Beauville structure in which one of the two triples has all elements of order $p^2$.
\end{thm}

\begin{proof}
We divide our proof into three cases depending on the number of maximal subgroups whose $p$th powers coincide, and in every case, we take into account Corollary \ref{number of coincidence} and Remark \ref{modifying generators}. First of all, note that since $G$ has at most three maximal subgroups of exponent $p$ and $p \geq 5$, there are at least three maximal subgroups of exponent $p^2$. 

\textbf{Case $1$:} Assume that there is a 1-1 correspondence between maximal subgroups $M_i$ of exponent 
$p^2$ and $M_i^p$. Choose a set of generators $\{x_1,y_1\}$ such that $o(x_1)=o(y_1)=o(x_1y_1)=p^2$.

\textbf{Case $2$:} Assume that there exist three maximal subgroups of exponent $p^2$ such that their $p$th power subgroups coincide. Then choose a set of generators $\{x_1,y_1\}$ such that $x_1, y_1$ and $x_1y_1$ fall into those maximal subgroups.

In both Case 1 and 2, since $p\geq5$, we can choose another set of generators $\{x_2,y_2\}$ so that each pair of elements in $\{x_i,y_i, x_iy_i \mid i=1,2\}$ is linearly independent modulo $G'$ by Remark \ref{modifying generators}.

\textbf{Case $3$:} Assume that we are not in the first two cases. Then there exist two maximal subgroups $M_1$, $M_2$ of exponent $p^2$ such that $M_1^p=M_2^p$ and $M^p\neq M_1^p$ for all other maximal subgroups $M$.

Let us first deal with $p\geq 7$. We start by choosing a set of generators $\{x_1, y_1\}$ where 
$x_1 \in M_1$ and $y_1 \in M_2$ are such that $o(x_1y_1)=p^2$, say $x_1y_1 \in M_3$. Then there might be a maximal subgroup $M_4$ such that $M_3^p=M_4^p$ (note that there is no other 
$i\neq 3,4$ satisfying $M_i^p=M_3^p$, otherwise we are in Case $2$). Since $p\geq 7$, we can choose another set of generators $\{x_2,y_2\}$ so that $x_2,y_2, x_2y_2 \notin M_4$ and each pair of elements in 
$\{x_i,y_i, x_iy_i \mid i=1,2\}$ is linearly independent modulo $G'$, by Remark \ref{modifying generators}.

If $p=5$ then by using the construction of metabelian thin $p$-groups in \cite{BCS} and the computer algebra system GAP, one can show that there is no metabelian thin $5$-group of class $5$ such that $|\gamma_5(G)|=5^2$ and $5$th powers of maximal subgroups coincide in pairs. Thus in Case $3$, there exists a maximal subgroup, say $M_3$, of exponent $5^2$, where all other $M^5$ are different from $M_3^5$. Then choose sets of generators $\{x_1, y_1\}$ and $\{x_2,y_2\}$ so that $x_1 \in M_1$ , $y_1 \in M_2$ and $x_1y_1 \in M_3$ and each pair of elements in $\{x_i,y_i, x_iy_i \mid i=1,2\}$ is linearly independent modulo $G'$.

We claim that, in every case, $\{x_1,y_1\}$ and $\{x_2,y_2\}$ form a Beauville structure for $G$. If $A=\{x_1,y_1,x_1y_1\}$ and $B=\{x_2,y_2,x_2y_2\}$, then we need to show that
\begin{equation}
\label{check for BS}
\langle a^g\rangle\cap \langle b^h\rangle=1,
\end{equation}
for all $a\in A$, $b\in B$ and $g,h\in G$. Note that $o(a)=p^2$ for every $a\in A$. Assume first that $o(b)=p$. If $\langle a^g\rangle\cap \langle b^h\rangle\neq 1$ for some $g,h\in G$, then $\langle b^h\rangle\subseteq\langle a^g\rangle$, and hence $\langle aG'\rangle=\langle bG'\rangle$, which is a contradiction, since $a$ and $b$ are linearly independent modulo $G'$. Thus we assume that $o(b)=p^2$. If (\ref{check for BS}) does not hold, then $\langle (a^g)^p \rangle=\langle (b^h)^p\rangle$, which contradicts the choice of $b$.
\end{proof}

In order to deal with  the case $\cl(G)=p+1$, we need the following lemma.

\begin{lem}\textup{\cite[Lemma~4.2]{FJ} }
	\label{lifting structure}
	Let $G$ be a finite group and let $\{x_1,y_1\}$ and $\{x_2,y_2\}$ be two sets of generators of $G$.
	Assume that, for a given $N\trianglelefteq G$, the following hold:
	\begin{enumerate}
		\item$\{x_1N,y_1N\}$ and $\{x_2N,y_2N\}$ is a Beauville structure for $G/N$,
		\item $o(u)=o(uN)$ for every $u\in\{x_1,y_1,x_1y_1\}$.
	\end{enumerate}
	Then $\{x_1,y_1\}$ and $\{x_2,y_2\}$ is a Beauville structure for $G$.
\end{lem}

\begin{thm}
Let $G$ be a metabelian thin $p$-group with $\cl(G)=p+1$, where $p\geq 5$. Then $G$ has a Beauville structure.
\end{thm}

\begin{proof}
By Theorem \ref{class p},  $\overline{G}=G/ \gamma_{p+1}(G)$ has a Beauville structure in which one of the two triples has all elements of order $p^2$, i.e. they have the same order in both $G$ and 
$\overline{G}$. Then we can apply Lemma \ref{lifting structure} and thus $G$ is a Beauville group.
\end{proof}

We next analyze the case $\cl(G)=p$ and $|\gamma_p(G)|=p$. Recall that we have $\gamma_p(G)\leq G^p\leq \gamma_{p-1}(G)$, and thus there are two possibilities:
\begin{enumerate}
	\label{cases}
	\item $G^p= \gamma_{p-1}(G)$,
	\item  $G^p=\gamma_p(G)$.
\end{enumerate}

Observe that by Lemma \ref{order of G^p},
 $G^p$ cannot be a proper subgroup of $\gamma_{p-1}(G)$ of order $p^2$.

\begin{thm}
	Let $G$ be a group  in case (i). Then $G$ has a Beauville structure .
\end{thm}

\begin{proof}
	First of all, notice that  there exists a pair of generators $a$ and $b$ of $G$ such that $a^p$ and $b^p$ are linearly independent modulo $\gamma_p(G)$. By the Hall-Petrescu formula, we have 
	\[
	(a^tb)^p=a^{tp}b^pc_2^{\binom{p}{2}} \dots c_p^{\binom{p}{p}},
	\]
	where $c_j \in \gamma_j(\langle a^t,b\rangle)$. Since $\gamma_2(G)^p \leq \gamma_p(G)$, by Lemma \ref{place of G^p}, we get
	\[
	(a^tb)^p \equiv a^{tp}b^p \pmod{\gamma_p(G)}
	\]
	for $1\leq t \leq p-1$. Observe that, similarly to Lemma \ref{companion}, for every maximal subgroup $M$, $m\in M$ and $c\in G'$ , we have $(mc)^p \equiv m^p \pmod{\gamma_p(G)}$. It then follows that the power subgroups $M^p$ are all different modulo $\gamma_p(G)$.
	
	On the other hand, since $\overline{G}= G/ \gamma_p(G)$ is of class $p-1$, it is a regular $p$-group such that $|\overline{G}^p|=p^2$. According to Corollary 2.6 in \cite{FG}, $\overline{G}$ is a Beauville group since $p\geq5$. From the observation above, all elements outside $G'$ are of order $p^2$ in both $G$ and $\overline{G}$. Then we can apply  Lemma \ref{lifting structure} to conclude that $G$ is a Beauville group.
\end{proof}

\begin{thm}
\label{special case}
Let $G$ be a group in case (ii). Then $G$ has a Beauville structure if and only if it has at least three maximal subgroups of exponent $p$.
\end{thm}

\begin{proof}
If the number of maximal subgroups of exponent $p$ is less than three, then $\Omega_1(G)$ is contained in the union of at most two maximal subgroups. Since $|G^p|=p$, it then follows from Proposition 2.4 in \cite{FG} that $G$ has no Beauville structure.

On the other hand, if at least three maximal subgroups have exponent $p$, then we choose a triple in which all elements have order $p$. Since $p\geq 5$, we can choose another triple such that each pair of elements in the union of the two triples is linearly independent modulo $G'$. Then $G$ has a Beauville structure.
\end{proof}

\section*{Acknowledgments}
We would like to thank G. Fern\'andez-Alcober for helpful comments and suggestions. Also, the second author would like to thank the Department of Mathematics at the University of L'Aquila for its hospitality while this research was conducted.

\end{document}